\documentclass[11pt]{amsart}
\usepackage{amsmath,amssymb,amsthm,fullpage,cancel,enumerate}
\usepackage{hyperref}
\usepackage{amsrefs}
\usepackage{setspace}

\theoremstyle{definition}
\newtheorem{theorem}{Theorem}
\newtheorem{lemma}{Lemma}

\newtheorem{prop}{Proposition}

\DeclareMathOperator{\is}{is}

\DeclareMathOperator{\sh}{sh}

\title{ Polynomiality of some hook-length statistics}
\author{Greta Panova }
\thanks{Harvard University, Mathematics Department}

\begin{document}

\begin{abstract}
We prove a conjecture of Okada giving an exact formula for a
certain statistic for hook-lengths of partitions:
\begin{equation*}
\frac{1}{n!} \sum_{\lambda \vdash n} f_{\lambda}^2 \sum_{u \in
\lambda} \prod_{i=1}^{r}(h_u^2 - i^2) = \frac{1}{2(r+1)^2} \binom{2r}{
r}\binom{2r+2}{ r+1} \prod_{j=0}^{r} (n-j),
\end{equation*}
where $f_{\lambda}$ is the number of standard Young tableaux of shape
$\lambda$ and $h_u$ is the hook length of the square $u$ of the Young
diagram of $\lambda$. We also obtain other similar formulas.
\end{abstract}
\maketitle
\textbf{MSC classes:}	 05E10 (Primary), 05A19


\textbf{Journal:} \emph{The Ramanujan Journal}, \textsc{10.1007/s11139-011-9332-z}

\section{Introduction}

If $F$ is any symmetric function then define
\begin{equation}\label{poly}
\Phi_n(F) = \frac{1}{n!} \sum_{\lambda \vdash n} f_{\lambda}^2F(h_u^2:u\in \lambda),
\end{equation}
where the sum runs over all partitions $\lambda$ of $n$. Here $h_u$ denotes the hook length of the square $u$ in that partition and  $f_{\lambda}$ is the number of standard Young tableaux of shape $\lambda$, given by the hook-length formula,\cite{hookformula}, $\displaystyle f_{\lambda}=\frac{n!}{\prod_{u\in\lambda} h_u}.$ In \cite{StaHooks} Stanley proves that $\Phi_n(F)$ is a polynomial in  $n$. Following this theorem Soichi Okada conjectured an explicit formula (see \cite{StaHooks}).
\begin{theorem}\label{okadathm}(Okada's conjecture) For every integer $n \geq 1$ and every nonnegative integer $r$ we have that
\begin{equation}\label{okada_eq}
\frac{1}{n!} \sum_{\lambda \vdash n} f_{\lambda}^2 \sum_{u \in \lambda} \prod_{i=1}^{r}(h_u^2 - i^2) = \frac{1}{2(r+1)^2} \binom{2r}{ r}\binom{2r+2}{ r+1}  \prod_{j=0}^{r} (n-j).
\end{equation}
\end{theorem}
 The current note is devoted to proving this equation and similar results. In doing so we also prove a conjecture by G. Han from \cite{HanConj} and generalize his ``marked hook formula'' from \cite{HanON}.

\section{ Proof of Okada's conjecture }
Let $P_r(n) = \frac{1}{n!} \sum_{\lambda \vdash n} f_{\lambda}^2 \sum_{u \in \lambda} \prod_{i=1}^{r}(h_u^2 - i^2)$. Since $F_r(x_1,\dots,x_n) = \sum_{i=1}^{n} \prod_{j=1}^r (x_i^2 - j^2)$ is clearly symmetric in the variables $x_1,\dots,x_n$, we see that $P_r(n) = \frac{1}{n!} \sum_{\lambda \vdash n} f_{\lambda}^2 F_r(h_u^2:u\in \lambda)$ and so by \cite{StaHooks} it is a polynomial in $n$. In order to prove \eqref{okada_eq} then it suffices to show that the degree of the polynomial is less than or equal to $r+1$, and exhibit  \eqref{okada_eq} for $r+2$ values of $n$.

\begin{lemma}\label{values}
The values $1,\dots,r$ are roots of $P_r(n)$. For the values at $r+1$ and $r+2$ we have $$P_r(r+1) = \frac{1}{2(r+1)^2} \binom{2r}{ r}\binom{2r+2}{ r+1} (r+1)!\; , \quad P_r(r+2) = \frac{1}{2(r+1)^2} \binom{2r}{ r}\binom{2r+2}{ r+1} (r+2)!.$$
\end{lemma}
\begin{proof} If $1\leq n \leq r$ we have for every $\lambda \vdash n$ and every $u \in \lambda$ that $1\leq h_u \leq |\lambda| = n \leq r$, and so $\prod_{i=1}^r(h_u^2 -i^2) = (h_u^2 -1^2)\cdots(h_u^2 - h_u^2) \cdots (h_u^2 -r^2) = 0$. Hence for $n=1,\dots,r$ we get 
 $$P_r(n) = \frac{1}{n!} \sum_{\lambda \vdash n} f_{\lambda}^2 \sum_{u \in \lambda} \prod_{i=1}^{r}(h_u^2 - i^2) = \frac{1}{n!} \sum_{\lambda \vdash n} f_{\lambda}^2 \sum_{u \in \lambda} 0 =0.$$ 

Now let $n=r+1$. Let $\lambda \vdash r+1$ and consider the largest hook length in $\lambda$, that is, $h_{(1,1)} = \lambda_1 +\ell(\lambda)-1$, where $\ell(\lambda)$ denotes the number of parts of $\lambda$. If $h_{(1,1)} \leq r$, then for every $u \in \lambda$ we would have $h_u \leq h_{(1,1)} \leq r$ and as in the previous paragraph we will have $\sum_{u \in \lambda} \prod_{i=1}^{r}(h_u^2 - i^2) =0$. 
When $h_{(1,1)}\geq r+1$ we must have $h_{(1,1)}=r+1$ and all of $\lambda$ be within that hook, so $\lambda = (a+1,\underbrace{1,\dots,1}_{r-a})$ for some $r\geq a\geq 0$.
%
For such $\lambda$ we have by the hook-length formula, or by a simple bijection with subsets of $\{2,\dots,r+1\}$ of $a$ elements for the entries in $(1,2),\dots,(1,a+1)$ of standard tableaux of shape $\lambda$, that $f_{\lambda}=\frac{(r+1)!}{(r+1)a!(r-a)!} = \binom{r}{a}$. We also have that the only square $u$ with hook length greater than $r$ is $(1,1)$, and for it we have $\prod_{j=1}^r(h_{(1,1)}^2-j^2) = \prod_{j=1}^r(r+1-j)\prod_{j=1}^r(r+1+j)= \frac{(2r+1)!}{r+1}$. Thus we can compute
\begin{align*}
P_r(r+1) & =   \frac{1}{(r+1)!} \sum_{a=0}^r f_{(a+1,1,\dots,1)}^2 \frac{(2r+1)!}{r+1} \\
 &= \frac{1}{2(r+1)^2} \binom{2r+2}{r+1} (r+1)! \sum_{a=0}^r \binom{r}{a}^2 = \frac{1}{2(r+1)^2} \binom{2r}{r}\binom{2r+2}{r+1} (r+1)!,
\end{align*}
which also agrees with \eqref{okada_eq}.

Computing $P_r(r+2)$ is slightly more complicated, because there are two kinds of shapes $\lambda$ which contain squares of hook length at least $r+1$. Since the largest hook length is $h_{(1,1)}$ we need to consider the cases $h_{(1,1)}=r+2$ and $h_{(1,1)}=r+1$. The first one implies that $\lambda$ is a hook, i.e. $(a+1,1,\ldots,1)$, and the only hook of length at least $r+1$ is at $(1,1)$ unless $a=0$ or $a=r+1$, when there are additional hooks of length $r+1$ at $(2,1)$ and $(1,2)$ respectively. Hence the contribution to $P_r(r+2)$ will be
\begin{align}
\begin{aligned}
\frac{1}{(r+2)!}\frac{(2r+2)!}{r+2} \sum_{a=0}^{r+1} &\binom{r+1}{a}^2 + \nonumber \\
&+ \frac{f_{(1,1,1,\dots,1)}^2}{(r+2)!} \prod_{j=1}^{r}((r+1)^2-j^2) + \frac{f_{(r+2)}^2}{(r+2)!} \prod_{j=1}^{r}((r+1)^2-j^2)\nonumber
\label{hook r+2}
\end{aligned}\\
 = \frac{(2r+2)!}{(r+2)(r+2)!} \binom{2r+2}{r+1} + 2\frac{(2r+1)!}{(r+1)(r+2)!}.
\end{align}

Next, if $h_{(1,1)}=r+1$ then $\lambda$ must necessarily be $(a+2,2,1,\dots,1)$ for some $a\in [0,\dots,r-2]$. In this case $h_{(1,1)}=r+1$ and all other hook lengths are less than $r+1$, so contribute 0 to $F_r$. Hence $F_r = \prod_{j=1}^{r}((r+1)^2-j^2)$.
We have by the hook-length formula and some algebraic manipulations of binomial coefficients that $$f_{(a+2,2,1,\dots,1)} = \frac{(r+2)!}{(r+1)(a+2)a! (r-a)(r-a-2)!} = (r+2)\binom{r}{a+1} - \binom{r+2}{a+2}. $$

Now we can compute the contribution of such partitions to the sum in $P_r(r+2)$ as
\begin{align}
\frac{1}{(r+2)!} &\sum_{a= 0}^{r-2} f_{(a+2,2,1,\dots,1)}^2 \prod_{j=1}^{r}((r+1)^2-j^2) \nonumber \\
&=\frac{(2r+1)!}{(r+2)!(r+1)}\sum_{a=0}^{r-2} \left( (r+2)\binom{r}{a+1} -\binom{r+2}{a+2} \right)^2  \nonumber \\
\label{hook r+1}
&=\frac{(2r+1)!}{(r+2)!(r+1)}  \left( (r+2)^2 \binom{2r}{r} - 2(r+2) \binom{2r+2}{r+1} + \binom{2r+4}{r+2} - 2 \right).
\end{align}
Finally, we obtain $P_r(r+2)$ as the sum of \eqref{hook r+2} and \eqref{hook r+1}. After some algebraic manipulations we get the desired
\begin{equation*}
P_r(r+2) =
\binom{2r}{r}\binom{2r+2}{r+1} \frac{1}{2(r+1)^2} (r+2)!.\qedhere
\end{equation*}
\end{proof}

\begin{lemma}\label{degree}
Let $\displaystyle R_k(n) = \frac{1}{n!}\sum_{\lambda \vdash n} f_{\lambda}^2 \sum_{u \in \lambda}h_u^{2k}$, then we have $\deg P_k(n)\leq \deg R_k(n) \leq k+1$ as polynomials in $n$.
\end{lemma}
\begin{proof}   The idea for this proof is suggested by Richard Stanley. The point is to use the bijection given by the RSK algorithm between pairs of standard Young tableaux $(P,Q)$ of same shape $\lambda \vdash n$ and permutations $w \in S_n$ (see for example \cite{EC2}), together with some permutation statistics. We are going to show that $0\leq \lim_{n \rightarrow \infty} \frac{R_k(n)}{n^{k+1}} < \infty$.

By the fact that the number of pairs $(P,Q)$ of SYT's of the same shape $\sh(P)=\sh(Q)=\lambda \vdash n$ is $f_{\lambda}^2$ and then by the RSK algorithm between such pairs and permutations of $n$ letters, we can rewrite $R_k(n)$ as
\begin{align}
R_k(n) &=
 \frac{1}{n!} \sum_{\lambda \vdash n} \sum_{\substack{(P,Q),\\ \sh(P)=\sh(Q)=\lambda}
} \sum_{u\in \lambda}h_u^{2k}  \nonumber \\
&= \frac{1}{n!} \sum_{\substack{(P,Q),\\ \sh(P)=\sh(Q) \vdash n}} \sum_{u \in \sh(P)} h_u^{2k} 
= \frac{1}{n!} \sum_{w \in S_n} \sum_{u \in \sh(w) } h_u^{2k},
\end{align}
where $\sh(w)$ denotes the shape of the SYTs obtained from $w$ by the RSK algorithm, i.e., if $(P_w,Q_w) = RSK(w)$, then $\sh(w) = \sh(P_w)=\sh(Q_w)$.

 We have that $h_{(1,1)} = \lambda_1 +\lambda'_1 -1$. Since for any $u\in \lambda$, $h_u \leq h_{(1,1)}$, and since for any $x,y\geq 0$, we have $(x+y)^m\leq (\max(x,y)+\max(x,y))^m=2^m(\max(x,y))^m \leq 2^mx^m+2^my^m$,  we have also that
$$h_u^{2k} \leq (\lambda_1+\lambda_1' -1)^{2k} \leq 2^{2k}\lambda_1^{2k} + 2^{2k} (\lambda_1')^{2k}.$$
 By Schensted's theorem, \cite{EC2}*{Cor.7.23.11}, we have that $\lambda_1 =\is(w)$, where $\is(w)$ denotes the length of the longest increasing subsequence of $w$.
 Hence $R_k(n)$ can be bounded as follows:
\begin{align}\label{bound}
R_k(n) &=
\frac{1}{n!} \sum_{w \in S_n} \sum_{u \in \sh(w)} h_u^{2k}
 \leq \frac{1}{n!} \sum_{\lambda \vdash n} f_{\lambda}^2 n(2^{2k}\lambda_1^{2k} + 2^{2k} (\lambda'_1)^{2k}) \nonumber \\
&= \frac{1}{n!} n2^{2k} \sum_{\lambda \vdash n} f_{\lambda}\lambda_1^{2k} +\frac{1}{n!} n2^{2k} \sum_{\lambda' \vdash n} f_{\lambda'} (\lambda'_1)^{2k} \nonumber\\
&= 2^{2k+1}n \frac{1}{n!} \sum_{\lambda \vdash n} f_{\lambda}\lambda_1^{2k}
 = 2^{2k+1}n\frac{1}{n!} \sum_{w\in S_n} \is(w)^{2k},
\end{align}
where we also used the obvious fact that $f_{\lambda'}=f_{\lambda}$, so that the sums over $\lambda$ and $\lambda'$ become equal.

Now that we have bounded $R_k(n)$ by sums involving only permutations, we can apply some permutations statistics to obtain bounds for these sums. In \cite{Hammer}*{Theorem 4} Hammersley proves that for uniformly distributed $w\in S_n$, the value $\is(w)/\sqrt{n}$ converges to a constant $c$ in probability and also in $L_p$ norm for any $p$. In other words for any $p>0$ there is a constant $E_p$ such that
\begin{equation} \label{expected}
\lim_{n \rightarrow \infty} \sum_{w \in S_n} \frac{1}{n!} \left( \frac{\is(w)}{\sqrt{n}} \right)^p = E_p.
\end{equation}
Thus the $L_p$ norm (also called $p^{\text{th}}$ moment) of $\frac{\is(w)}{\sqrt{n}}$ is bounded. In other words for any nonnegative $k$ there is a constant $M_k$ such that
$$\frac{1}{n^{k/2}}\sum_{w \in S_n} \frac{1}{n!}\is(w)^k<M_k.$$
By this fact and by the bounds in \eqref{bound} we see that
\begin{align*}
R_k(n) \leq 2^{2k+1}n \sum_{w \in S_n} \frac{1}{n!} \is(w)^{2k} \leq 2^{2k+1}n M_{2k} n^k = 2^{2k+1}M_{2k} n^{k+1},
\end{align*}
so that we must necessarily have that $\deg R_k(n) \leq k+1$ for every $k$.  Since $\prod_{j=1}^{k}(h_u^2 - j^2) \leq h_u^{2k}$ we have that $P_k(n) \leq R_k(n)$, so $\deg P_k(n) \leq \deg R_k(n)$, and in particular $\deg P_k(n) \leq k+1$. 
\end{proof}
\begin{proof}[Proof of Theorem \ref{okadathm}]
In Lemma \ref{values} we showed that $P_k(n)$ coincides with the polynomial in $n$ $$\frac{1}{2(k+1)^2} \binom{2k}{k} \binom {2k+2}{k+1} \prod_{j=0}^{k}(n-j)$$ of degree $k+1$ at $k+2$ values, so since $\deg P_k(n) \leq k+1$ the two polynomials should agree. Hence we have that
$$ P_k(n) = \frac{1}{2(k+1)^2} \binom{2k}{k} \binom {2k+2}{k+1} \prod_{j=0}^{k}(n-j), $$
proving Okada's conjecture \eqref{okada_eq}. 
\end{proof}

This theorem  shows, in particular, that $\deg P_k(n) = k + 1$ and so from Lemma \ref{degree} we must have $\deg R_k(n) = k+1$, proving a conjecture of Han, \cite{HanConj}*{Conjecture 3.1}.

We observe now that Okada's conjecture Theorem~\ref{okadathm} gives us a formula for $\Phi_n(p_k)$, where $p_k$ are the power sum symmetric functions given by $p_k(x_1,\dots,x_n) = x_1^k+\dots+x_n^k$, or in other words a formula for $R_k(n)$. 
We will express $p_k(x)=x^k$ as a linear combination of $q_i(x) =\prod_{j=1}^{i} (x - j^2)$ as follows.
Consider the central factorial numbers $T(k,i)$(see exercise 5.8 in \cite{EC2}), given by
$T(k,i) = i^2 T(k-1,i) + T(k-1,i-1)$ and $T(0,0)=1$, $T(i,j)=0$ if $i=0,j>0$ or $i> 0,j=0$. We have 
\begin{align}\label{p_rec}
p_k =\sum_{i=0}^kT(k+1,i+1)q_i,
\end{align}
since by induction on $k$ we get
\begin{align*}
p_{k+1} &= xp_k = \sum_{i=0}^kT(k+1,i+1)xq_i = 
\sum_{i=0}^kT(k+1,i+1)(x-(i+1)^2 +(i+1)^2)q_i\\
&= \sum_{i=1}^{k+1}T(k+1,i)q_{i} + \sum_{i=0}^k (i+1)^2T(k+1,i+1)q_i =\sum_{i=0}^{k+1}T(k+2,i+1)q_i.
\end{align*}
%
Equations \eqref{p_rec} and \ref{okada_eq} give the following proposition which generalizes Han's "marked hook formula" for $\Phi_n(p_1)$, \cite{HanON}*{Theorem 1.5}.
\begin{prop}
For $\Phi_n(p_k)$ we have that
\begin{align*}
\Phi_n(p_k) = \frac{1}{n!} \sum_{\lambda \vdash n} f_{\lambda}^2 \sum_{u \in \lambda} h_u^{2k} =&
\sum_{i=0}^k T(k+1,i+1) \Phi_n(q_i) = \\
&\sum_{i=0}^k T(k+1,i+1) \frac{1}{2(i+1)^2} \binom{2i}{i} \binom{2i+2}{i+1} (i+1)! \binom{n}{i+1}.
\end{align*}
\end{prop}

We will now exhibit a  more general  upper bound for the degree of $\Phi_n(p_{\mu})$, where $\mu = (\mu_1,\dots,\mu_j)\vdash k$ with $\mu_j \neq 0$ and we use the power sum symmetric function $p_{\mu}$ as $F$.
In this case we have that
\begin{align*}
\Phi_n(p_{\mu}) &= \frac{1}{n!} \sum_{\lambda \vdash n} f_{\lambda}^2 \left( \sum_{u \in \lambda} h_u^{2\mu_1} \right) \dots \left( \sum_{ u \in \lambda} h_u^{2\mu_j} \right) \\
&\leq \frac{1}{n!} \sum_{\lambda \vdash n} f_{\lambda}^2 ( n \max( h_u: u \in \lambda)^{2\mu_1})\cdots (n \max( h_u: u \in \lambda)^{2\mu_j}) \\
&=n^j \frac{1}{n!} \sum_{\lambda \vdash n} f_{\lambda}^2 \max( h_u: u \in \lambda)^{2k}\leq n^jR_k(n).
\end{align*}
From Lemma \ref{degree} we have $\deg R_k(n)\leq k+1$, so we get that $$\deg \Phi_n(p_{\mu}) \leq j+k .$$

\section{Other similar results}

Consider now the case of $F = e_k$, where $e_k$ is the elementary symmetric function given by $e_k(x_1,\dots,x_n) = \sum_{1\leq i_1 <i_2<\dots < i_k \leq n} x_{i_1}x_{i_2}\dots x_{i_n}$. We will show how to find a formula for $\Phi_n(e_k)$. The point is to use the Okounkov-Nekrasov hook length formula \cite{HanON},
\begin{equation}\label{on}
\sum_{n \geq 0} \sum_{\lambda \vdash n}x^n \prod_{u \in \lambda} \left(1 - \frac{z}{h_u^2} \right) =
\prod_{k\geq 1} (1 - x^k)^{z-1}.
\end{equation}
We should point out that the same approach has already been used by Han in \cite{HanEu} to derive the cases for $e_1$ and $e_2$ and the following is an extension of his results.

If we make the substitution $1/z = t$ and $y = x/t$ and expand the product over $u$ in the left-hand side of \eqref{on} we obtain
\begin{equation}
\sum_{n \geq 0} \sum_{\lambda \vdash n} y^n \prod_{u \in \lambda} \frac{1}{h_u^2} \left(\sum_{j = 0}^{n} e_{j}(\{h_u^2:u \in \lambda\})t^j(-1)^{n-j} \right) = \prod_{k \geq 1} \left(1 - (yt)^k\right)^{1/t -1}.
\end{equation}

Substituting $\frac{1}{n!}\sum_{\lambda \vdash n} f_{\lambda}^2 e_j(h_u:u \in \lambda)$ with $\Phi_n(e_j)$ we get
\begin{equation}\label{on2}
\sum_{n \geq 0} \frac{y^n}{n!} \sum_j (-1)^{n-j}\Phi_n(e_j)t^j = \prod_{k \geq 1} (1-(yt)^k)^{1/t -1}.
\end{equation}
So the value of $\Phi_n(e_j)$ is $(-1)^{n-j}n!$ times the coefficient of $y^nt^j$ from the right hand side of \eqref{on2}. We will now expand the right-hand side in a convenient form as follows
\begin{align}\label{on_rhs}
\prod_{k \geq 1} (1-(yt)^k)^{1/t -1} &=
 \exp\left( \left(1 -1/t \right) \left( \sum_{k \geq 1} -\log(1-(yt)^k)\right)\right) \nonumber \\
 &= \exp\left(\left( 1 - 1/t \right) \left( \sum_{k \geq 1, i \geq 1} (yt)^{ki}\right) \right) 
= \exp\left( \left(1-1/t\right) \left( \sum_{m \geq 1} (yt)^m \tau(m)\right) \right) \nonumber \\
&= \sum_{u \geq 0} \frac{\left(1-1/t\right)^u}{u!} \sum_{m_1,\dots,m_{u} \geq 1} (yt)^{m_1+\cdots+m_u} \tau(m_1)\cdots \tau(m_u) ,
\end{align}
where $\tau(m)$ is the number of divisors of $m$. Restricting \eqref{on_rhs} to the coefficient at $y^n$ is equivalent to imposing the condition $m_1+\dots+m_u=n$, then restricting further to $t^j$ is equivalent to taking only the term at $\left(\frac{1}{t}\right)^{n-j}$ from $(1-1/t)^u$, so we get that
\begin{equation}
\Phi_n(e_j) = n!(-1)^{n-j} \sum_{u = 0}^{ n}\frac{1}{u!} \binom{u}{n-j} (-1)^{n-j} \sum_{\substack{m_1+\cdots+m_u = n,\\ m_i \geq 1}} \tau(m_1)\cdots \tau(m_u).
\end{equation}
Notice that in order for $\binom{u}{n-j} \neq 0$ we would need $ u \geq n-j$, so we can write $q = n -u$, going from $0$ to $j$, and then further substitute $m_i = a_i +1$, $a_i \geq 0$, so that $\sum_{i = 1}^{n-q} a_i = q$. Thereby we get that
\begin{align*}
\Phi_n(e_j) = \sum_{q = 0}^{j} \frac{n!}{(n-j)! (j-q)!} \sum_{\substack{a_1+\cdots+a_{n-q} = q, \\ a_i \geq 0}} \tau(a_1+1)\cdots\tau(a_{n-q}+1).
\end{align*}
Notice that the unordered solutions $(a_1,\dots,a_{n-q})$ of $a_1+\dots+a_{n-q}=q$ are in bijection with the choice of $p \leq q$ of the $a_i$s to be nonzero. If we label those nonzero $a_i$s by $b_k$s for $k=1,\dots,p$ we obtain the following
\begin{prop}
\begin{align*}
\Phi_n(e_j)= \binom{n}{j} \sum_{q = 0}^{j} \frac{j!}{(j-q)!} \sum_{p=0}^{q} \binom{n-q}{p} \sum_{b_1+\cdots+b_p = q, b_i \geq 1} \tau(b_1+1)\cdots \tau(b_p+1).
\end{align*}
\end{prop}

As expected, $\Phi_n(e_j)$ is indeed a polynomial in $n$ for any $j$.

\end{document}